%% file: main.tex
\documentclass[a4paper,12pt]{article}
\usepackage{amsfonts,amsthm,amsmath,amssymb, comment}
\usepackage[backgroundcolor=white,linecolor=black]{todonotes}
\usepackage{enumitem}
\usepackage[utf8]{inputenc}
\usepackage[unicode]{hyperref}
\usepackage[english]{babel}
\usepackage[notcite,notref,final]{showkeys} 
\usepackage[left=1cm,right=3cm, top=3cm,bottom=1.5cm,bindingoffset=0cm, marginparwidth=3cm]{geometry}

\newcounter{theorems}
\numberwithin{theorems}{section}
\newtheorem{thm}[theorems]{Theorem}
\newtheorem{cor}[theorems]{Corollary}
\newtheorem{lem}[theorems]{Lemma}
\newtheorem{prop}[theorems]{Proposition}

\newtheorem*{thm*}{Theorem}
\theoremstyle{definition}
\newtheorem{defin}[theorems]{Definition}
\theoremstyle{remark}
\newtheorem{ex}[theorems]{Example}
\newtheorem{rem}[theorems]{Remark}
\newtheorem*{rem*}{Remark}

\def \ci {{S^1}}
\def \eps {\varepsilon}
\def\e {\varepsilon }
\def\SS {\Sigma }
\def\Sig {{\Sigma^{s}} }

\def\om {\omega }

\def\oms {{\omega_{stat} }}
\def\As {{A_{stat} }}
\def\attr {{\mathcal A}}
\def\диффео {diffeomorphism }

\def\be {\begin{equation}}
\def\ee {\end{equation}}
\def\RR {\mathbb R}
\def\NN {\mathbb N}
\def\KV {\mathcal I}

\def \bi {\begin{itemize}}
\def \ei {\end{itemize}}
\def \be {\begin{enumerate}}
\def \ee {\end{enumerate}}
\DeclareMathOperator{\dist}{dist}
\DeclareMathOperator{\Diff}{Diff}
\DeclareMathOperator{\orb}{{\overline{orb}}}
\DeclareMathOperator{\oorb}{{orb}}

\author{A.V. Okunev\footnote{National Research University Higher School of Economics}, I.S. Shilin\footnote{Moscow State University}}
\title{On the attractors of step skew products over the Bernoulli shift\thanks{Supported in part by the RFBR (grant 16-01-00748 a).}}
\date{}

\begin{document}
\baselineskip=14pt

\maketitle
\input{intro.tex}

\input{definitions.tex}
\input{bones.tex}

\input{proj.tex}
\input{circle.tex}
\input{semicont.tex}

\section{Acknowledgements}
We would like to thank Yu. S. Ilyashenko and S.S. Minkov for useful discussions and helpful comments.
\newpage

\input{biblio.tex}
\newpage
\end{document}

%% file: intro.tex

\abstract{The statistical and Milnor attractors of step skew products over the Bernoulli shift are studied. For the case of the fiber a circle we prove that for a topologically generic step skew product the statistical and the Milnor attractor coincide and are Lyapunov stable. For this end we study some properties of the projection of the attractor onto the fiber, which might be of independent interest. For the case of the fiber being a segment we give a description of the Milnor attractor as the closure of the union of graphs of finitely many almost everywhere defined functions from the base of the skew product to the fiber.

\smallskip
\noindent \textbf{Keywords:} skew products, attractors, Lyapunov stability.


\section{Introduction}
\subsection{The ``microverse'' of dynamical systems}
The space of step skew products over the Bernoulli shift with one-dimensional fiber\footnote{We consider only skew products whose fiber maps preserve orientation; this condition is assumed in most of the results below.} is a ``microverse'', of sorts, where all sorts of interesting dynamical phenomena can be observed. Rough properties discovered in the class of step skew products can (at least, sometimes) be recreated in the space of diffeomorphisms of smooth manifolds --- this can be done via the so called \emph{Gorodetski--Ilyashenko strategy}. This strategy was used in~\cite{GI99}, \cite{KN}, \cite{Ily2}, \cite{Kud} and in a number of other works. 

Among the phenomena discovered in the class of step skew products with the fiber a segment is the existence of so called \emph{bony attractors} (\cite{Kud}). A bony attractor is an attractor that is \emph{a bony graph}, and a bony graph is a union of a \emph{graph} of an almost everywhere defined function from the base of the skew product to the fiber and some \emph{bones}, which are segments in the fibers over those points of the base where this function is not defined. 

There are also two genuinely surprising phenomena in the class of skew products with the fiber a segment whose fiber maps preserve the boundary of the segment. The first one is called \emph{intermingled basins of attraction} (\cite{Kan}, see also~\cite{BM}, \cite{IKS}, \cite[\S11.1.1]{BDV}). The point is that there exists a boundary preserving map of a cylinder such that it is a skew product over the circle doubling with the fiber a segment and Lebesgue almost every point is attracted either to the upper or to the lower boundary of the cylinder, basins of both boundaries being everywhere dense. Note that in this example the attractor (the union of two boundary circles) is Lyapunov unstable, because arbitrarily close to the upper boundary circle there are points attracted to the lower one. Perturbing this example in the class of boundary preserving maps of a cylinder,  one can obtain an open set of maps with the same properties. 
The second unexpected phenomenon is the local typicality of maps with \emph{thick} attractors~(\cite{Ily}, \cite{Ily2}), i.e., attractors that have positive but not full Lebesgue measure. 

Besides these examples we should also mention several general results on skew products with one-dimensional fiber. V.~Kleptsyn and D.~Volk proved in~\cite{KV} that for typical\footnote{for an open and everywhere dense set, actually.} step skew products with the fiber a segment there are finitely many ``attractor'' bony graphs (possibly, without bones) that attract almost all points of the phase space, with the exception of finitely many ``repeller'' bony graphs. They also showed that there exists a finite number of $SRB$-measures such that the union of their basins has full Lebesgue measure.
M.~Viana and J.~Yang~\cite{VY} proved the latter property for a wide class of partially hyperbolic diffeomorphisms with one-dimensional central foliation; this class includes partially hyperbolic skew products with the fiber a segment or a circle.

The properties of step skew products with the fiber a circle are in a sense similar to the properties of skew products with the fiber a segment. Namely, for a typical step skew product with circle fiber, either it is minimal (i.e., the orbit of any point of the circle under the action of the semigroup generated by the fiber maps is dense), or all fiber maps have a common absorbing domain that is a union of finitely many segments (\cite{KKO}, in preparation). 
 
We are interested in the following two questions about the attractors of dynamical systems posed by Yu.~S.~Ilyashenko:
\begin{itemize}
\item Is there an open set of diffeomorphisms with Lyapunov unstable attractors?
\item Is there an open set of diffeomorphisms with ``thick'' attractors (i.e., attractors that have positive but not full measure)?
\end{itemize}
In the present work we will discuss the statistical and Milnor attractors. These definitions were introduced in~\cite{AAISh} and~\cite{Mil} respectively, and we recall them in section~\ref{s:attractors} below. 

In the case of arbitrary diffeomorphisms of compact manifolds (of dimension 2 or greater) Lyapunov instability of the attractors is a locally topologically generic phenomenon: a locally residual set of systems with Lyapunov unstable attractors can be found in the so called Newhouse domains~(\cite{Shi}). Yet the question about the existence of an open set of diffeomorphisms with unstable attractors is still open. 

As we mentioned earlier, in the class of boundary-preserving step skew products with the fiber a segment, there are open domains where skew products have unstable~(\cite{Kan}) or thick~(\cite{Ily}) attractors. However, the boundary preservation requirement does not look natural. Therefore, a question arises whether such domains exist if we do not require the boundary to be mapped into itself. It turns out that there are none.  

For a typical step skew product with the fiber a segment the attractor has zero measure --- this was proved in~\cite{KV} (see also~\cite{KV2}). A typical step skew product with the fiber a circle is, by the dichotomy from~\cite{KKO}, either minimal (then the attractor is the whole phase space) or has an absorbing domain (then the attractor has zero measure, which is deduced from the analogous result for the segment case).     

\subsection{The main results}

Let us state the main results of the paper. The required definitions can be found in section~\ref{s:definitions} below.

\paragraph{SSPs with one-dimensional fiber.}
We will consider two classes of step skew products (SSPs) over the Bernoulli shift --- SSPs with the fiber a circle such that the fiber maps are diffeomorphisms, and SSPs with the fiber a segment such that the fiber maps map the segment into itself and are diffeomorphisms on the image. For both classes we assume that the fiber maps preserve the orientation and the space of fiber maps is endowed with the $C^r$-topology (for some $r \ge 1$), which defines the topology on the space of SSPs. We shall say that a property is \emph{topologically generic} in one of these classes if it holds for a residual subset of SSPs in this class. 

\begin{itemize}
\item
\emph{For a topologically generic step skew product with the fiber a circle or a segment the statistical attractor is Lyapunov stable and coincides with the Milnor attractor (Theorem~\ref{t:stable} and Corollary~\ref{c:interval}).}

\item[] 
A similar statement for smooth skew products over Anosov diffeomorphisms is proved in~\cite{Oku}. Note that it is unknown whether the Milnor attractor is (generically) asymptotically stable even for step skew products with the fiber a segment. 

\item
\emph{For an open and dense set of SSPs with the fiber a segment the Milnor attractor is the union of the closures of the graphs of some almost everywhere defined functions from the base to the fiber (Theorem~\ref{AM_thm}).} 
These graphs were introduced in~\cite{KV}. It also follows from~\cite{KV} that the statistical attractor equals the same union. If the statistical attractor is Lyapunov stable, it can be proved that it coincides with the Milnor attractor. But we prove Lyapunov stability not for an open and dense set of SSPs, but only for a residual one.
\end{itemize}

\paragraph{SSPs with arbitrary fiber.}
To prove that attractors are stable we use the following general properties of attractors of SSPs which may be of independent interest. These properties are proved for the general case: the fiberwise maps can be diffeomorphisms of any compact manifold, possibly with boundary. For manifolds with boundary we consider diffeomorphisms of those manifolds onto themselves.

\begin{itemize}
\item
\emph{The statistical or Milnor attractor of any step skew product over the Bernoulli shift can be reconstructed by its projection onto the fiber. Namely, the point is in the attractor if and only if the projection of its whole past semi-orbit lies in the projection of the attractor (Theorem~\ref{t:proj})).}

\item
\emph{For a step skew product the statistical attractor or the Milnor attractor is Lyapunov stable if and only if its projection onto the fiber is Lyapunov stable in the sense of Definition~\ref{d:stproj} (Theorem~\ref{stproj}).}

\end{itemize}

%% file: definitions.tex
\section{Preliminaries} \label{s:definitions}

\subsection{Step skew products} \label{s:SKP}
Consider the set $\Sig = \{1, \dots, s\}^{\mathbb Z}$ of biinfinite sequences $\om = \dots \om_{-1} \om_{0} \om_{1} \dots $ of symbols $1, \dots, s$. 
For two distinct sequences $\om, \tilde\om\in \Sig$ define the distance between them as
\[
d(\om, \tilde\om) = 2^{-m}, \; m = \min\{|n|\colon \om_n\ne\tilde\om_n\}.
\]
Given $m\in\NN$ different integers $n_1, \dots, n_m$ and $m$ symbols $\alpha_1, \dots, \alpha_m$ one can define \emph{a cylinder} in $\Sig$ as follows:
\[
CS_{\alpha_1,\dots, \alpha_m}^{n_1,\dots, n_m} = \{\omega\in \Sig \mid \omega_{n_j} = \alpha_{j}, \; j = 1, \dots, m\}.
\]
Such cylinders generate a topology on $\Sig$, and therefore they also generate the corresponding Borel $\sigma$-algebra over $\Sig$.

The \emph{$(1/s, \dots, 1/s)$-Bernoulli measure} $\mu_\Sig$ on $\Sig$ is defined in the following way. First, define it on cylinders by the formula
\[
\mu_\Sig(CS_{\alpha_1,\dots, \alpha_m}^{n_1,\dots, n_m}) = 1/s^m;
\]
then continue it to the whole Borel $\sigma$-algebra, and, finally, continue $\mu_\Sig$ to the corresponding Lebesgue $\sigma$-algebra. Note that this measure is a probability measure.

\begin{rem}
Our proofs work for any Bernoulli measure provided that the probability of every symbol is positive.  For simplicity we restrict our arguments to the case of equiprobable symbols.
\end{rem}

The map
$\sigma\colon \Sig \to\Sig, \; (\sigma\om)_n = \om_{n+1},$
is called \emph{the Bernoulli shift}. It is not difficult to check that this is a homeomorphism that preserves measure $\mu_\Sig$.

\begin{defin}
\emph{A step skew product} (SSP for short) over the Bernoulli shift $(\Sig, \sigma)$ with the \emph{fiber} $M$ and the \emph{fiber maps} $f_1, \dots, f_s \colon M \to M$ is a map $F$ from a space $X=\Sig \times M$ to itself that has the following form: 
\begin{equation} \label{e:SKP}
F:X \rightarrow X,\; (\omega, p) \mapsto (\sigma \omega, f_{\omega_0}(p)).
\end{equation}
Here $\omega_0$ is the symbol at the zero position in the sequence $\omega$. 
\end{defin}

Here are a few notes regarding this definition.

\begin{enumerate}\setlength\itemsep{-0.1em}
\item The space $\SS$ is called \emph{the base} of the SSP, whereas $M$ is called \emph{the fiber}. In what follows, $M$ will be a compact manifold with Riemannian metric. 

\item The metric on $X$ is obtained as the sum of the distances along the fiber and along the base. 

\item On $X$ there is a measure $\mu_X$ obtained as the product of measure $\mu_\Sig$ on the base and the Lebesgue measure $\mu_M$ on the fiber. 

\item An SSP is uniquely determined by its fiber maps. Thus SSPs with the base $\Sig$, the fiber $M$ and $C^r$-smooth fiber maps form a metric space isomorphic to $(C^r(M))^s$. We will also work with various subsets of this space (e.g., $(\Diff^r(M))^s$) with the induced from $(C^r(M))^s$ topology.

\end{enumerate}

\subsection{Milnor and statistical attractors} \label{s:attractors}

Consider a dynamical system $(X, F)$ where $X$ is a separable metric space and $F:X \to X$  is a continuous map. Fix a finite Borel measure $\mu$ on $X$. In the case of SSPs later on $\mu$ will always be the product $\mu_X$ of the Lebesgue measure in the fiber and the Bernoulli measure in the base.

\begin{defin}[\cite{Mil}] 
\emph{The Milnor attractor} of a map $F$ is the smallest closed subset of $X$ that contains $\omega$-limit sets of $\mu$-a.e. points. 
\end{defin}

We will denote the Milnor attractor of a map $F$ by $A_M(F)$ or simply by $A_M$ if it is clear which map is considered.

\begin{defin}
\emph{The frequency} $Freq(x, U)$ with which the orbit of a point $x$ visits the set $U$ is the upper limit
\[
\limsup_{N \to +\infty} \frac 1 N \#\{n: F^n(x) \in U, \; 0 \le n < N \}.
\]
\end{defin}
\begin{defin}
\emph{The statistical $\omega$-limit set} of a point $x\in X$ (notation: $\oms(x)$) is the set of points $z \in X$ such that for any neighborhood $U$ of $z$ one has $Freq(x, U)>0$.
\end{defin}

The statistical attractor is defined exactly like the Milnor attractor, but with statistical $\omega$-limit sets instead of the regular ones. 
\begin{defin}[\cite{AAISh}, $\S8.2$; see also~\cite{GI}] 
\emph{The statistical attractor} is the smallest closed subset of $X$ that contains $\oms(x)$ for $\mu$-a.a. points $x\in X$. Notation:  $\As(F)$ or $\As$.
\end{defin}
The definition of the statistical attractor in~\cite{AAISh} is slightly different from the one we gave, but is equivalent to it.

The existence of Milnor attractors is proved in~\cite[Lemma~$1$]{Mil} for the case when $F$ is a continuous map of a compact manifold to itself and $\mu$ is the Lebesgue measure on this manifold. The existence of the Milnor and statistical attractor for a skew product with measure $\mu_X$ is proved in exactly the same way (provided that the fiber is compact).
Milnor and statistical attractors are both forward invariant (and also backward invariant if~$F$ is a homeomorphism), because for any point~$x$ the sets $\om(x)$ and $\oms(x)$ are invariant. 

\begin{rem} \label{p:defAM}
A point $x$ belongs to $A_M$ iff for any its neighborhood $U \ni x$ there is a positive measure set of points $y$ such that $\omega(y)$ intersects~$U$. The same is true for $\As$, but the regular $\omega$-limit sets are replaced by the statistical ones.
\end{rem}

\subsection{Maximal attractors and Lyapunov stability} \label{s:Amax}

We will also use the definition of a maximal attractor.

\begin{defin}
Let $U$ be an absorbing domain for the map~$F$, i.e., an open set such that $\overline{F(U)}~\subset~U$. \emph{The maximal attractor} in the domain~$U$ is the set 
$$A_{max}(F, U) = \bigcap \limits_{n=1}^\infty F^n(U).$$ 
\end{defin}

The attractor of the inverse map is called \emph{the repeller}.

The following (a priori non-strict) inclusions always hold:
$$A_{stat}(F) \subset A_M(F) \subset A_{max}(F, X).$$
The first one follows from the fact that for any point $x$ one has $\oms(x)\subset\om(x)$. The second one holds because the maximal attractor of the dissipative domain~$U$ always includes $\om(x)$ for any $x\in U$. 

\begin{defin}
A set $A$ is called \emph{absorbing} for a map $F$ if $F(A) \subset A$.
\end{defin}

\begin{defin}
An invariant or absorbing closed subset $M$ of the phase space $X$ of the system $(X, F)$ is called \emph{Lyapunov stable} if for any its neighborhood~$U$ there exists a smaller neighborhood~$V$ of $M$ such that (positive semi-) trajectories that start inside~$V$ never quit~$U$.
\end{defin}

\begin{defin} \label{d:stproj}
For a skew product of type~\eqref{e:SKP} a closed subset of the fiber $B \subset M$ will be called \emph{Lyapunov stable} if the set $\Sigma^s \times B$ is Lyapunov stable.
\end{defin}

It follows from the definition of the maximal attractor that it is always Lyapunov stable. At the same time Milnor and statistical attractors can be unstable: an example is given by a map of a circle with a unique fixed point which is semi-stable, e.g.,
$$
x \mapsto x + 0.1 (1-\cos x), \; x \in \mathbb R/2\pi\mathbb Z. 
$$
For an arbitrary $x$ one has $\om(x) = \oms(x) = \{0\}$; hence $A_M = \As = \{0\}$. However, on one of the two sides the points run away from zero, which makes the attractors Lyapunov unstable. 

%% file: bones.tex
\section{Milnor attractors of SSPs with a segment fiber}\label{sect:KV}
\subsection{Preliminaries and the statement of the result}
  The results of this section are true for a wider class of skew products than the one defined in section~\ref{s:SKP}; namely, they are true for step skew products over a transitive topological Markov chain~$(\SS, \sigma)$ with finitely many states. In this case one should fix on the base an ergodic Markov measure~$\mu_\SS$ such that for this measure all admissible transitions have positive probability (see the exact definitions in~\cite[Sect.~2]{KV} or~\cite[\S 4.2 e]{KH}). However, the reader might as well assume that in this section we deal with SSPs over the Bernoulli shift with the Bernoulli measure on the base.

We are considering SSPs with the fiber a segment, i.e., maps of the form
\begin{equation}
F\colon X=\Sigma\times I \to X, \;
(\omega,x){\mapsto}(\sigma\omega,f_{\omega_0}(x)),\label{basic}
\end{equation}
where all fiber maps $f_{\omega_0}$ are orientation preserving diffeomorphisms of a segment onto its image.

\begin{defin}
A closed subset $K\subset X$ is called \emph{a bony graph} if it intersects $\mu_\SS$-a.e. fiber by a single point and intersects the rest of the fibers by a (non-degenerate) segment. Such segments are called \emph{bones}.
\end{defin}
Note that a bony graph can be viewed as a union of its bones and a graph of some almost everywhere defined map from the base to the fiber, hence the name. The Fubini theorem implies that the $\mu_X$-measure of a bony graph is zero.

\begin{defin}
A subset $\Pi \subset X$, bounded by two graphs of continuous mappings from the base to the fiber, is called \emph{a strip}. A strip is \emph{strictly trapping} if $\overline{F(\Pi)} \subset \Pi$ and \emph{strictly inverse trapping} if $\overline{\Pi} \subset F(\Pi)$.
\end{defin}

\begin{thm}[V.~A.~Kleptsyn, D.~S.~Volk, \cite{KV}]\label{KV}
There exists an open and dense (in any $C^r$-topology for $r \ge 1$) subset $\KV$  of the set of all step skew products of type~\eqref{basic} with the fiber a segment such that for any SSP $F \in \KV$ the following holds.
\begin{enumerate}\setlength\itemsep{-0.1em}
\item The phase space can by covered by a union of finitely many strictly trapping and strictly inverse trapping strips.

\item \label{bony_graph} The maximal attractor in every trapping strip is a bony graph; the same is true for the repellers inside the inverse trapping strips.

\item\label{SRB_measure} For every strip there is a unique ergodic invariant measure that projects to the Markov measure in the base. This measure may be obtained by lifting the Markov measure from the base to the maximal attractor (or repeller) of the strip, viewed as the graph of an almost everywhere defined measurable function. This measure is an SRB-measure inside the corresponding strip\footnote{i.e., for $\mu_X$-almost-every  point $p$ from the strip the positive semi-orbit is distributed according to this measure. Recall that the positive semi-orbit of $p$ is distributed according to a measure $\nu$ if the sequence of measures $\frac{1}{n}\sum_{j=1}^{n-1}\delta_{F^j(p)}$ weakly converges to $\nu$.};

\item\label{Lyapexp} The fiberwise Lyapunov exponents of those measures are negative for trapping strips and positive for inverse trapping ones.

\end{enumerate}
\end{thm}

Note that the set $\KV$ is given by some explicit conditions. They can be found in section~5 of~\cite{KV}.

For $F \in \KV$, let us denote by $\Pi_i \; (i = 1,\dots, l)$ the dissipative strips and by $\Gamma_i \subset \Pi_i \; (i = 1,\dots, l)$ the graphs of almost everywhere defined functions from the base to the fiber which, together with the corresponding bones, form the maximal attractors of the dissipative strips.

\bigskip

\begin{thm}\label{AM_thm}
For any $F \in \KV$ the Milnor attractor is the closure of the union of the graphs~$\Gamma_i$.
\end{thm}

All points of every inverse trapping strip, except the points of the corresponding maximal repeller, leave this strip under the iterates of~$F$ and enter one of the trapping strips. Since bony graphs have zero measure, for almost all points of the phase space their positive semi-orbits enter one of those trapping strips. Therefore, it suffices to consider an arbitrary trapping strip $\Pi \in \{\Pi_i\}_{i=1}^n$ and its attracting graph $\Gamma \in \{\Gamma_i\}$ of the almost everywhere defined function $\phi$ from the base to the fiber and prove the following statement.

\begin{lem}\label{AM_stripe}
For any $F \in \KV$ for any $\Pi \in \{\Pi_i\}_{i=1}^n$ the Milnor attractor of the restriction $F|_\Pi$ is the closure of the graph~$\Gamma$.
\end{lem}

\subsection{Plan of the proof of Lemma~\ref{AM_stripe}}
Let us show that the closure of the graph $\Gamma$ is contained in the Milnor attractor.
Property~\ref{SRB_measure} of Theorem~\ref{KV} says that the closure of $\Gamma$ is the support of the $SRB$-measure according to which the orbits of almost all points of the strip are distributed. It follows that for almost all points $x\in\Pi$ one has $\oms(x) = \overline{\Gamma}$, and therefore $\overline{\Gamma} = A_{stat}\subset A_M$.

It remains to show that almost all points of the strip are attracted to the graph~$\Gamma$. First, using the negative fiberwise Lyapunov exponent (property~(\ref{Lyapexp} in Theorem~\ref{KV}), we will find a set $V\subset\Pi$ of positive measure such that every point in this set is attracted to the graph (Lemma~\ref{l:V} below). The set $V$ is obtained with the help of the Egorov theorem, and we will call it ``the Egorov set'' sometimes. To construct this set, we will consider the subsets $U_\alpha\subset \Pi$ that are covered when one moves the graph~$\Gamma$ up and down along the fiber to a distance up to~$\alpha > 0$. The set~$V$ will have the form $V = U_\alpha\cap(B\times I)$, where $\alpha$ is sufficiently small and $B\subset \SS$ is a subset of measure~$1-\delta$.

Then we will show that almost all points of the strip~$\Pi$ visit the set~$V$ (Proposition~\ref{p7}). This is proved in the following way. The projection of the set~$V$ onto~$\Sigma$ coincides with~$B$ and has measure $1-\delta$. In addition to that, we will show (using property~\ref{bony_graph} in Theorem~\ref{KV}) over some set $C\subset\Sigma$ of measure $1-\delta$ the images of the boundaries of our trapping strip converge uniformly to each other. Fix a number $N$ such that for every $n > N$ the distance along the fiber between the $n$-th images of the boundaries of the stripe is less than $\alpha$ over $C$. Then, since $V = U_\alpha\cap(B\times I)$ and $\Gamma\subset F^n(\Pi)$, we will have
\begin{equation} \label{e:B-cap-C}
F^n(\Pi) \cap ((B\cap C)\times I) \subset V.
\end{equation}
Almost all points of the strip~$\Pi$ visit the set $((B\cap C)\times I)\cap\Pi$ infinitely many times. If the number of the iterate for which the visit happens is greater than $N$, the point finds itself inside~$V$ by~\eqref{e:B-cap-C}. Hence, almost every point of the strip~$\Pi$ is attracted to the closure of the graph~$\Gamma$, that is, $A_M(F|_{\Pi}) \subset \overline{\Gamma}$. The inverse inclusion was already proved above, therefore $A_M(F|_{\Pi}) = \overline{\Gamma}$. Now we proceed to the detailed proof.

\subsection{Constructing the Egorov set \texorpdfstring{$V$}{V}}

Denote by~$X_\Gamma$ the union of those fibers that contain a point of the graph~$\Gamma$. Let $\rho$ be the almost everywhere defined function from~$X_\Gamma$ to $\mathbb R$ that gives the fiberwise distance from its argument to the graph~$\Gamma$:
$$\mbox{for}\; p=(\omega,x)\in X_\Gamma \; \; \rho(p):={\rm dist}(p, \Gamma\cap(\{\omega\}\times I)).$$
One can also regard~$\rho$ as a function that is defined almost everywhere on~$X$.

Consider a family of sets $\mathcal U = \{U_{\alpha}\}_{\alpha\in\mathbb{R}_{+}}$ such that $U_{\alpha}$ is defined as follows:
$$U_\alpha = \{p\in X_\Gamma\mid \rho(p) \le \alpha\}.$$
In other words, the set $U_\alpha$ is covered when one moves the graph~$\Gamma$ up and down along the fiber by a distance up to~$\alpha$. In what follows we will consider only small~$\alpha$ for which one has $U_{\alpha}\subset \Pi.$

\begin{lem}\label{l:V}
For arbitrarily small~$\delta>0$ there are a set $B\subset\SS$ of measure $1-\delta$ and a number $\gamma>0$ such that for for any point~$p$ in the set $V = U_\gamma\cap(B\times I)$ one has $\om(p)\subset\overline{\Gamma}.$ 
\end{lem}

\begin{proof}
By property~\ref{Lyapexp} from Theorem~\ref{KV}, the fiberwise Lyapunov exponent~$L$ of the ``attracting'' SRB-measure on $\overline{\Gamma}$ is negative. Since this measure is obtained by lifting the measure~$\mu_\SS$ onto $\Gamma$, we have 
 $$L=\int\limits_{\Sigma}{}{\rm ln}f_{\omega}'(x_A(\omega))d\mu_\SS,$$
where $x_A(\omega)$ is the~$x$-coordinate of the intersection of the graph~$\Gamma$ with the fiber $\{\omega\}\times I.$ Note that the function we want to integrate is defined almost everywhere on~$\Sigma$ and is bounded, and thus integrable.

Fix a small number $\varepsilon > 0$ such that
$$L_{\varepsilon}:=\int\limits_{\Sigma}{}{\rm ln}(f_{\omega}'(x_A(\omega))+\varepsilon)d\mu_\SS < 0.$$

Let $g_\varepsilon(\omega):= f_{\omega}'(x_A(\omega))+\varepsilon$. Denote by $K_n(\om)$ the time averages of the function~$\ln g_\eps$ at a point $\om\in\SS$:
$$
K_n(\omega):=\frac{1}{n}\sum_{k=0}^{n-1}\ln g_\eps(\sigma^k\omega) = \frac{1}{n}\ln\prod_{k=0}^{n-1}g_\eps(\sigma^k\omega).
$$

Since the fiber is compact, all fiber maps are uniformly continuous. Their number is finite, therefore given the number $\e$ from above one can find a number $\beta=\beta(\varepsilon)>0$ such that for any $x_1, x_2\in I$ and~$j$ the following implication holds\begin{equation}\label{imp}
 {\rm dist}(x_1,x_2)<\beta \; \Rightarrow \; |f_{j}'(x_1)-f_{j}'(x_2)|<\varepsilon.
\end{equation}
Fix this~$\beta$ and consider the set $U_\beta\in \mathcal U$. If the positive semi-orbit  of a point of the phase space lies inside~$U_{\beta},$ we can estimate the speed with which it approaches the graph.

\begin{prop}\label{p5}
If the point $p = (\omega, x)$ and its images under the first $n-1$ positive iterates of~$F$ lie inside~$U_{\beta}$, then
\begin{equation}\label{est}
\rho(F^n(p))\le \rho(p)\cdot\prod_{k=0}^{n-1}g_\varepsilon(\sigma^k\omega) = \rho(p)\cdot\exp(nK_n(\om)).
\end{equation}
\end{prop}

\begin{proof}
Suppose $p=(\omega, x) \in U_\beta.$ Let $x_A(\om)$, as above, be the~$x$-coordinate of the point of the graph~$\Gamma$ inside the fiber over $\omega$. Then~\eqref{imp} implies that for $t\in [x_A(\omega)-\beta, \; x_A(\omega)+\beta]$ we have $f'_{\omega_0}(t) < f_{\omega_0}'(x_A(\omega))+\varepsilon$, which means that we can write the estimate
$$\rho(F(p)) = |f_{\omega_0}(x) - f_{\omega_0}(x_A(\omega))| \le |x - x_A(\omega)|\cdot(f_{\omega_0}'(x_A(\omega))+\eps) =\rho(p)\cdot g_\eps(\omega).$$

Similarly, if for every $j\in\{1,\dots, n-1\}$ we have $F^j(p)\in U_\beta$, this yields an estimate
$$\rho(F^n(p))\le \rho(F^{n-1}(p))\cdot g_{\e}(\sigma^{n-1}\omega) \le $$ $$\le\rho(F^{n-2}(p))\cdot g_{\e}(\sigma^{n-2}\omega)\cdot g_{\e}(\sigma^{n-1}\omega)\le\dots\le \rho(p)\cdot\prod_{k=0}^{n-1}g_\varepsilon(\sigma^k\omega).$$
\end{proof}

Now we can finish the proof of Lemma~\ref{l:V}.

By Birkhoff's theorem, for $\mu_\SS$-a.e. points $\omega\in\SS$ one has ${K_n(\omega)\to L_{\varepsilon}}$ as~$n\to\infty.$
Take an arbitrary $\delta\in (0, \frac{1}{2})$. Applying the Egorov theorem\footnote{The Egorov theorem says that if on a space~$\Sigma$ with a probability measure there is an almost everywhere convergent sequence of measurable functions, then for any $\delta>0$ there exists a set $B\subset\Sigma$ of measure at least $1-\delta$ such that this sequence converges uniformly on $B$.} to the sequence~$K_n$, we obtain a set $B\subset\Sigma$ of measure $1-\delta$ such that on $B$ this sequence uniformly converges. This set $B$ is the one that appears in the statement of Lemma~\ref{l:V}. The uniform convergence of $K_n$ to $L_\eps$ on $B$ and inequality $L_\eps < 0$ imply that there is an integer $M\in\mathbb N$ such that for any~$n\ge M, \; \omega\in B$ we have $K_n(\omega)<0.$

Let $\lambda$ be a constant such that all fiber maps are~$\lambda$-Lipschitz. Choose $\gamma$ to be so small that $\gamma\cdot \lambda^M < \beta.$ Now let us show that for any point $p\in V = U_\gamma\cap(B\times I)$ we have $\omega(p)\subset\overline{\Gamma}$.

Indeed, for a point $p = (\om, x)\in V$ we have $\rho(p)<\gamma.$ Then for $n\le M$ the following inequality holds $$\rho(F^n(p))<\gamma\cdot \lambda^M < \beta.$$ 
This means that during the first~$M$ iterates the images of the point~$p$ will not leave~$U_{\beta}$. Hence we can use estimate~\eqref{est}. Since $K_M(\omega)<0$ when $\omega\in B,$ we have: 
$$\rho\left(F^M(p)\right)\le \rho(p)\cdot \exp(M\cdot K_M(\omega))\le\rho(p),$$
i.e., the $M-$th image of the point~$p$ also lies in~$U_{\gamma}.$ Since for any $n\ge M$ we have $K_n(\omega)<0$, the subsequent images of~$p$ also stay in~$U_{\gamma}$, where estimate~\eqref{est} is applicable. Thus, for any $n>0$ we have
$\rho(F^n(p)) \le \rho(p)\cdot \exp(n\cdot K_n(\omega))$. Since $K_n(\omega)\to L_{\varepsilon} < 0$ as $n\to\infty$, we conclude that $\rho(F^n(p))\to 0$.
This means that $\omega(p)\subset \overline{\Gamma}$.

The proof of Lemma~\ref{l:V} is complete. 
\end{proof}

\subsection{Almost all points enter the Egorov set \texorpdfstring{$V$}{V}}

\begin{prop}\label{p7}
For almost every point~$x\in\Pi$ there exists a positive integer~$n$ such that $F^n(x)\in V$.
\end{prop}

\begin{proof}
Suppose that the boundary of the trapping strip~$\Pi$ is formed by the graphs of two continuous maps~$\psi_1$ and~$\psi_2$ from the base to the fiber. Then the set $F^n(\Pi)$ lies between the~$n$-th images of these graphs. Those images are themselves graphs of continuous maps $\psi_{1,n},\; \psi_{2,n}$ from the base to the fiber. When $n\to\infty$, the difference $\psi_{1,n} - \psi_{2,n}$ tends to zero $\mu_\SS$-almost-everywhere, otherwise the maximal attractor of the strip~$\Pi$ would have positive $\mu_X$-measure.

By the Egorov theorem, over some set $C\subset\Sigma$ of $\mu_\SS$-measure~$1-\delta$ this convergence is uniform, which means that there exists $N$ such that when $n \ge N$ we have
\begin{equation}\label{eq:psis}
\|(\psi_{1,n} - \psi_{2,n})|_C\|_{C^0}<\gamma.
\end{equation}

Denote $D=B\cap C$ (recall that $B$ is the same as in the definition of~$V$). Since $\mu_\Sigma(B) = \mu_\Sigma(C) = 1-\delta$,  the measure of~$D$ is at least $1-2\delta.$ Since the measure~$\mu_\SS$ is ergodic for the shift~$\sigma$, $\mu_\SS$-a.e. point $\omega\in\Sigma$ visits~$D$ infinitely many times under the iterates of~$\sigma$. But if for a point~$p\in\Pi$ its image~$F^n(p)$ finds itself inside the set~$D\times I$ when $n\ge N$, then inequality~\eqref{eq:psis} implies that $F^n(p)\in V.$ Thus, almost every point of the dissipative strip gets inside $V$ after iterating $F$ sufficiently many times, which is exactly what we were to prove.
\end{proof}

Lemma~\ref{l:V} and proposition~\ref{p7} together prove Lemma~\ref{AM_stripe}: Proposition~\ref{p7} says that almoust all orbits intersect~$V$, whereas Lemma~\ref{l:V} says that all points in~$V$ are attracted to the graph~$\Gamma$. As we mentioned earlier, Lemma~\ref{AM_stripe} implies Theorem~\ref{AM_thm}. Thus, for typical SSPs of the class under consideration the Milnor and the statistical attractors coincide with the closure of the union of the attracting graphs~$\Gamma_j$.

\begin{rem}
Similarly Lemma~\ref{AM_stripe} one can prove that for any SSP $F \in \KV$ and any inverse trapping strip~$\Pi$ the attractor of $F^{-1}|_{\Pi}$ coincides with the closure of the corresponding graph~$\tilde{\Gamma}_j$. If we agree to call the smallest closed set that contains $\alpha$-limit sets of almost all points for which those are defined \emph{the Milnor repeller}, then the Milnor repeller for~$F$ would be the closure of the union of the repelling graphs~$\tilde\Gamma_j$.
\end{rem}

\subsection{A counterexample to a weakened version of Theorem~\ref{AM_thm}}
The proof of Theorem~\ref{AM_thm} works for all SSPs satisfying statements~1-4 of Theorem~\ref{KV}. The most important one is statement~4 that claims that the Lyapunov exponent is negative. But it seems that this statement is not necessary, and statements~1-3 are sufficient. Indeed, one may try to argue as follows. Denote by $\tilde{\Pi}$ the intersection of our strip~$\Pi$ with the union of fibers that contain the points of the graph. The maximal attractor of the restriction of our dynamical system to~$\tilde{\Pi}$ is the graph~$\Gamma$. All points of $\tilde{\Pi}$ are attracted to its maximal attractor, therefore almost all points of~$\Pi$ are attracted to the graph.

Unfortunately, this argument is flawed: we can not claim that all points from $\tilde\Pi$ are attracted to the graph. Consider the following example.

\begin{ex} \label{e:example}
Let $F\colon X = \Sigma^2\times [-1, 1]\to X$ be an SSP over the Bernoulli shift with fiber maps that satisfy the following conditions.
\begin{itemize}
\item{Both fiber maps send the segment $[-1, 1]$ strictly inside itself;}
\item{zero is the only fixed point for~$f_1$.}
\item{$f_1(0) = f_2(0) = 0, \; f'_2(0) = 2, \; f'_1(0) = 1/2$;}
\item{the restrictions of both maps $f_1, f_2$ to some segment $[-\eps, \eps]$  are linear;}
\end{itemize}
Note that this example is degenerate.

\begin{lem} \label{l:example}
For the map from Example~\ref{e:example} $A_{max}(F, X)$ is a bony graph that consists of the section $\Gamma = \{(\om, x)\mid x = 0\}$ and some set of bones that has zero measure. Moreover, $\As(F) = \overline\Gamma = \Gamma$. However, $A_M(F)\ne \overline\Gamma$.
\end{lem}

\noindent One may also prove that there exists an $SRB$-measure according to which the orbits of almost all points are distributed. This is the measure $\mu_{\Sigma^2} \times \delta(0)$, the product of the Bernoulli measure in the base and the delta-measure at zero in the fiber. Its support coincides with $\As(F)$ and with $\overline\Gamma$. 

Let us give the idea of the proof of Lemma~\ref{l:example}.

\begin{itemize}
\item Let us begin with the third statement. First let us show that $A_M(F)\not\subset \Gamma$. In the logarithmic charts on intervals $(0, \eps]$ and $[-\eps, 0)$ our dynamical system is just a symmetric random walk.  It follows from the properties of such a random walk that the orbits of almost all points of the strip $\Pi_0 =\Sigma^2\times [-\eps, \eps]$ leave this strip for both forward and backward iterates of $F$ (at this point we are not even interested whether they get back eventually). Hence the set~$B_0$ of points of the strip $\Pi_0$ which never leave~$\Pi_0$ {\it and} are attracted to~$\Gamma$ under the iterates of~$F$ has zero measure. Therefore, the whole basin of attraction of the section~$\Gamma$ has zero measure, because it is just the union $\cup_{n \in \mathbb N} F^{-n}(B_0)$.

\item Let us move on to the first statement. The inclusion $\overline \Gamma\subset A_{max}(F, X)$ is obvious. 

To prove that $A_{max}(F, X)$ is a bony graph it suffices to show that $\mu_X(A_{max}(F, X)) = 0$.  Note that the point $p\in X$ belongs to $A_{max}(F, X)$ if and only if the preimage $F^{-n}(p)$ is defined for every~$n>0$. On the other hand, there is $n_0>0$ such that the $f_1^{n_0}$-preimage is not defined for any point of the fiber that lie outside the segment $[-\eps,\eps]$. Thus for a random point $p\in X\setminus\Pi_0$ the preimage $F^{-n_0}(p)$ is not defined with probability at least $2^{-n_0}$. Since almost all points eventually leave the strip~$\Pi_0$ when we repeatedly take the $F$-preimages, it is not difficult to derive that for almost all points of $X$ some $F$-preimage is not defined, and therefore $\mu_X(A_{max}(F, X)) = 0$.

\item Now let us prove the second statement. We prove by contradiction that the measure $\delta_0$ is the only probability stationary measure on~$I$ for the couple of fiber maps $f_1, f_2$ applied with equal probability.\footnote{Recall that a measure $\nu$ is called \emph{stationary} if for any measurable $A\subset I$ one has $\nu(A) = \frac{1}{2}(\nu(f^{-1}_1(A)) + \nu(f^{-1}_2(A)))$.} Suppose there is a stationary measure~$\nu$ such that $\mathrm{supp}(\nu)\ne\{0\}$. Since its support $\mathrm{supp}(\nu)$ must be forward invariant under $f_1$, this support should intersect the set $[-\eps, \eps]\setminus\{0\}$. Then there exists a segment $J\subset [-\eps, \eps]\setminus \{0\}$ such that $\nu(J)>0, \; J\cap f_1(J) = \emptyset$. Since measure $\nu$ is stationary and $f_1 = f_2^{-1}$ in $[-\eps, \eps]$, for any $j\in \mathbb N$ we have $\nu(f_1^j(J)) = \frac{1}{2}\Big( \nu(f^{j-1}_1(J)) + \frac{1}{2} \nu(f^{j+1}_1(J))\Big)$, i.e., $\Big(\nu(f_1^j(J))\Big)_{j\in\mathbb N}$ is an arithmetic progression. Since $\nu$ is a probability measure, all members of this arithmetic sequence are nonnegative and its sum is at most~$1$. Hence this is a series of zeros, which contradicts our assumption that $\nu(J) > 0$.  

\item For any $x\in I$, for $\mu_{\Sigma^2}$-a.e. $\omega\in\Sigma^2$ any partial limit (in $\ast$-weak topology) of the sequence ${\frac{1}{n}\sum_{j = 0}^{n-1} \delta(f_{\omega_j}\circ\dots\circ f_{\omega_0}(x))}$ is a stationary probability measure~(\cite[p.~4]{Doin}, see also~\cite[Lemma~2.5]{Fur}).

\item Since $\frac{1}{n}\sum_{j = 0}^{n-1}\delta(\sigma^j(\omega)) \to \mu_{\Sigma^2}$ for almost all $\omega$, two previous bullet points imply that for the SSP $F$ positive semi-orbits of almost all points of $X$ are distributed according to the measure~$\mu_{\Sigma^2}\times\delta(0)$. This means that $\As(F) = \Gamma$.
\end{itemize}

\end{ex}

%% file: proj.tex
\section{On the projection of the attractor onto the fiber} \label{s:proj}
In this section we will consider SSPs over the Bernoulli shift of form~\eqref{e:SKP}. The fiber will be an arbitrary compact manifold $M$ (possibly with boundary), and the fiber maps will be its diffeomorphisms. For manifolds with boundaries here and below in this section we will consider diffeomorphisms onto the manifold. The projection onto the fiber along the base will be denoted as $\pi_M: X \to M$.

When proving the Lyapunov stability of the statistical attractor of a generic SSP with the fiber a circle (section~\ref{s:stable}) it will be more convenient to work not with the attractor itself but with its projection to the fiber. In the present section we will study connections between the attractor and its projection.
First of all we will provide a criterion, analogous to Remark~\ref{p:defAM}, for a point of the fiber to lie in the projection of the attractor. Then we will show that the attractor can be reconstructed using its projection to the fiber. After that we will prove Lyapunov stability of the projection of the attractor (see Definition~\ref{d:stproj}) is equivalent to the stability of the attractor itself.

\begin{rem} \label{r:AMAS}
All results of this section will stay true if statistical $\omega$-limit sets are replaced by the regular ones and the statistical attractor is replaced by the Milnor attractor. The proofs can be obtained by the ones presented below by the same substitution.
\end{rem}

\subsection{Reconstructing the attractor from its projection onto the fiber}
For a sequence~$\omega=\dots\omega_{-1} \omega_0 \omega_1 \dots$$\in \Sig$ let us call the sequence~$\omega_0 \omega_1 \omega_2\dots$ its \emph{future half} and the sequence~ $\dots \omega_{-2} \omega_{-1}$ its \emph{past half}. We shall need the following proposition, which is a counterpart of Remark~\ref{p:defAM} for the projection of the attractor onto the fiber.

\begin{prop} \label{p:defPAM}
The point~$x \in M$ belongs to the projection of~$\As$ onto the fiber iff for any open subset~$U \ni x$ the points $y\in X$ such that $\pi_M(\oms(y))$ intersects~$U$ form a set of positive measure.
\end{prop}
\begin{proof}
Suppose that a point $x\in M$ does not belong to the projection of the attractor onto the fiber. The set~$\pi(\As)$ is compact, since it is the image of a compact set~$\As$ under a continuous map, and therefore the point~$x$ has a neighborhood~$V\subset M$ that does not intersect with~$\pi(\As)$. Then the set~$\Sig \times V$ does not intersect~$\As$, which implies that for almost every point $y\in X$ the set~$\pi_M(\oms(y))$ does not intersect~$V$.

Now suppose that the point~$x$ is in the projection of the attractor, a point $z \in \As$ being projected to $x$. The required condition on~$x$ is obtained by applying Remark~\ref{p:defAM} to the point~$z$ and then considering the projections to the fiber.
\end{proof}

\begin{lem} \label{l:proj}
For $\mu_X$-almost-any point $y \in X$, the set $A=\oms(y)$ can be reconstructed from its projection to the fiber in the following way: the point $x$ belongs to $A$ if and only if the projections of all preimages of this point  belong to the projection of~$A$. 
\end{lem}

\begin{proof}
\noindent Fix a finite word $w \in \{1, \dots, s\}^{|w|}$ and an open set~$U_M \subset M$. 

\noindent Let us introduce the following notation:
\bi
\item[] $U := \Sig \times U_M$,
\item[] $U_w := \{\tilde \omega \in \Sig:  \tilde \omega_0 \dots \tilde \omega_{|w|-1}=w\} \times U_M,$
\item[] $Y(w, U_M)$ is the set of all points $y \in X$ such that either the positive semi-orbit of~$y$ enters~$U$ only finitely many times or the lower limit of the ratio of times it spends in $U_w \subset U$ and in~$U$ is positive, i.e.,
$$
\liminf \limits_{N \to \infty} \frac{ \#\{ n<N: F^n(y) \in U_w \} }{ \#\{ n<N: F^n(y) \in U \} } > 0.
$$
\ei

\begin{prop} \label{p:proj}
For any word~$w$ and any open set $U_M \subset M$ the set $Y(w, U_M)$ has full measure. 
\end{prop}
\begin{proof}
It suffices to show that for any point~$p \in M$ we have 
\[
\mu_\Sig(\{\omega: (\omega, p) \in Y(w, U_M)\})=1,
\]
and then use the Fubini theorem.

Fix a point~$p \in M$.
Suppose that the $k$-th visit of the orbit of the point~$(\omega, p)$ to the set~$U$ happens at the time $t_k=t_k(\omega)$ (i.e., $F^{t_k}(\; (\omega, p) \;) \in U$). It is possible that there are only finitly many of those visits in total, so some $t_k$ can be undefined. Let $l(\omega)$ be the number of the first undefined $t_k$. 
 
Let us regard the base~$\Sig$ with the measure~$\mu_\Sig$ as a probability space and call subsets of~$\Sig$ \emph{events}.  
Let us define the following sequence of events on $\Sig$: 
 \[
 A_k = \{\omega: F^{t_k}(\; (\omega, p) \;) \in U_w\}, \; k \in \mathbb N.
 \]
Once again, some $A_k$ may be undefined. 

Let us assume first that $l(\omega) = \infty$ for every $\omega$.
Then we need to prove that the lower limit of the fraction of the events $A_k$ that happened is almost surely positive.
It follows from the definition of $U_w$ that~$\omega \in A_k$ if and only if the future half of the basewise coordinate of the point~$F^{t_k}(\; (\omega, p) \;)$ begins with the word~$w$. We can rewrite this as
\[
\omega_{t_k(\omega)} \dots \omega_{t_k(\omega)+|w|-1} = w.
\]

Consider first a particular case when~$|w|=1$. Fix $m>0$ and consider the $\sigma$-algebra~$\mathcal A_m$ generated by the events~$A_1, \dots, A_m$ and by the random variable~$t_{m+1}$. Then the conditional probability of the event~$A_{m+1}$ given $\mathcal A_m$ is constant and equal to~$1/s$ (recall that $\mu_\Sig$ id the $(1/s, \dots, 1/s)$-Bernoulli measure). This follows from the fact that for a fixed value of~$t_{m+1}$ the event~$A_{m+1}$ depends on the symbol~$\omega_{t_{m+1}}$, whereas the events~$A_1, \dots, A_m$ depend on the symbols of the sequence $\omega$ with numbers not greater than~$t_{m+1}-1$. 
Hence events $A_k$ are mutually independent and each of them happens with probability~$1/s$. By the strong law of large numbers, the limit of the fraction of the events~$A_k$ that happened almost surely exists and equals~$1/s$.

Now let the length of the word~$w$ be arbitrary. The previous argument does not work now because the subwords of the sequence~$\omega$ that define the events~$A_m$ and $A_{m+1}$ may overlap. This argument can be saved by applying it to the subsequence~$A_{|w|}, A_{2|w|}, \dots$. Then the numbers of the first letters of the words responsible for~$A_{m|w|}$ and $A_{(m+1)|w|}$ differ by at least~$|w|$, and therefore these subwords of~$\omega$ do not overlap. 
Hence the sequence~$A_{|w|}, A_{2|w|} \dots$ is formed by mutually independent events, each of which happens with probability~$\frac 1 {s^{|w|}}$. Applying the strong law of large numbers to this subsequence we conclude that the lower limit of the fraction of events~$A_k$ that happened is almost surely at least~$\frac 1 {|w| s^{|w|} }$.

Now let us get rid of the assumption that~$l(\omega) = \infty$. 
For this end let us define the analogues~$\tilde A_k$ of the events~$A_k$ on the probability space~$(\Sig \times \Sig, \mu_\Sig \times \mu_\Sig)$. 
Define~$\tilde A_k$ as the set formed by the pairs $(\omega, \omega') \in \Sig \times \Sig$ such that
\bi
\item either $t_k(\omega) < \infty$ and $\omega_{t_k(\omega)} \dots \omega_{t_k(\omega)+|w|-1} = w$,
\item or $t_k(\omega) = \infty$ and $\omega'_{k-l(\omega)} \dots \omega'_{k-l(\omega)+|w|-1} = w$.
\ei 

Arguing as above, we see that the subsequence~$\tilde A_{|w|}, \tilde A_{2|w|} \dots$ consists of mutually independent events. Thus the lower limit of the fraction of events~$\tilde A_k$ that happened is almost surely not smaller than~$\frac 1 {|w| s^{|w|} }$. This implies that the set $\tilde B \subset \Sig \times \Sig$ of pairs $(\omega, \omega')$ such that the fraction of the events~$\tilde A_k$ that happened tends to zero has zero $\mu_\Sig \times \mu_\Sig$-measure. Denote by~$B \subset \Sig$ the set of $\omega$ such that every event $A_k$ (without tilde) is defined and the limit of the fraction of the events~$A_k$ that happened equals zero. Since for any $\omega \in B$ one has $\omega \times \Sig \subset \tilde B$, the Fubini theorem implies that $\mu_\Sig(B)=0$. Therefore, for almost every $\omega \in \Sig$
\bi
\item either only a finite number of events $A_k$ is defined 
\item or the lower limit of the fraction of the events $A_k$ that happened is positive. 
\ei
Thus, $(\omega, p) \in Y(w, U_M)$ for almost every $\omega$.
\end{proof}

Let us continue the proof of Lemma~\ref{l:proj}. We need to construct a full measure set~$Y$ such that for any~$y \in Y$ for the set~$\oms(y)$ the following would hold:
the point $x$ belongs to $\oms(y)$ if and only if the projections onto the fiber of all preimages of~$x$ lie in the projection of~$\oms(y)$. 

In order to construct~$Y$, fix a countable base of topology~$\{U_i\}_{i\in \mathbb N}$ on $M$ and set 
\[
Y=\bigcap_{w, i} Y(w, U_i),
\]
where $w$ ranges over all finite words and $i$ ranges over all positive integers. By Proposition~\ref{p:proj} all sets $Y(w, U_i)$ have full measure. Thus,~$Y$, being their countable intersection, is a set of full measure as well.

Now let's prove that for $y \in Y$ the set~$\oms(y)$ has the required property. The ``only if'' part follows from the invariance of the set~$\oms(y)$, and the ``if'' part is yet to be proved.

Consider a point 
$$x=(\omega, p),\; \omega \in \Sig,\; p \in M.$$
Since $\oms(y)$ is closed, in order to check that $x \in \oms(y)$ it suffices to prove that~$\oms(y)$ intersects any cylindrical neighborhood~$V$ of the point~$x$. 

Let a cylindrical neighborhood have the form $V=V_\Sig \times V_M$, where 
$V_\Sig \subset \Sig$ is a set of sequences~$\tilde \omega$ such that $\tilde \omega_{-n} \dots \tilde \omega_{n-1}=\omega_{-n} \dots \omega_{n-1}$, 
whereas $V_M \subset M$ is a neighborhood of the point~$p$ in the fiber.

Since the set $\oms(y)$ is invariant, it would be sufficient to show that this set intersects~$F^{-n}(V)$. It is not difficult to see that $F^{-n}(V)= U_w \times U_M$ where 
$U_w$ is the set of all sequences~$\tilde \omega \in \Sig$ such that $\tilde \omega_0 \dots \tilde \omega_{2n-1}=\omega_{-n} \dots \omega_{n-1}$, and
$$U_M=(f^{-1}_{\omega_{-n}} \circ \dots \circ f^{-1}_{\omega_{-1}})(V_M).$$

From the countable base of topology used in the definition of the set~$Y$, take any set $U_i$ such that \begin{equation} \label{e:Ui}
\pi_M F^{-n}(x) \in U_i \subset U_M.
\end{equation}
Since we assume that the projections to the fiber of all preimages of the point~$x$ lie in the projection of~$\oms(y)$, the point $\pi_M F^{-n}(x)$ lies in $\pi_M \oms(y)$. Together with~\eqref{e:Ui} this means that~$U_i$ intersects~$\pi_M \oms(y)$. Therefore, the set~$\oms(y)$ intersects~$\Sig \times U_i$. The following property can be easily deduced from the definition of the set~$Y(w, U_i)$:
\begin{itemize}
\item[] Let $z\in Y(w, U_i)$ and let $\oms(z)$ intersect $\Sig \times U_i$. Then~$\oms(z)$ intersects~$U_w \times U_i$.
\end{itemize}
\noindent Applying this property to $z=y$, we conclude that~$\oms(y)$ intersects $U_w \times U_i \subset U_w \times U_M = F^{-n}(V)$.   
\end{proof}

\noindent We have proved Lemma~\ref{l:proj}. Let us use it to obtain some corollaries. 

\begin{cor} \label{c:future}
For almost every point~$y \in X$ it does not depend on the future half of the base coordinate of~$x$ whether $x$ belongs to the set $A=\oms(y)$ or not. More formally, if $(\omega, p) \in A$, then $(\tilde \omega, p) \in A$ for any sequence~$\tilde \omega$ that has the same past half as $\omega$. This statement is also true if~$A$ is the whole statistical attractor $\As$ instead of just the set $\oms(y)$.
\end{cor}
\begin{proof}
The first claim follows from Lemma~\ref{l:proj}, because the criterion given there does not use the future half of the base coordinate. 
The second one follows from the first due to Remark~\ref{p:defAM}.   

\end{proof}
\begin{rem*}
For any partially hyperbolic diffeomorphism the statistical $\omega$-limit set of Lebesgue almost every point is saturated by unstable fibers (this follows from~\cite[Theorem~$11.16$]{BDV}). The same is true for the regular $\omega$-limit sets (\cite{MO}). Corollary~\ref{c:future} is an analogue of these statements for the case of the SSPs.
\end{rem*}

\begin{cor} \label{c:inv}
For almost every point~$y \in X$ the projection onto the fiber of the set $A=\oms(y)$  is forward-invariant under the action of all fiber maps~$f_i$.
\end{cor}
\begin{proof} 
Let $p \in \pi_M(A)$. Consider an arbitrary point $x=(\omega, p) \in A$ projected into $p$. By replacing $\omega_0$ with the symbol $i$, we can get a point $\tilde x$ that lies in~$A$ by Corollary~\ref{c:future}. Therefore, $F(\tilde x)$ lies in~$A$ too. Since $\pi_M(F(\tilde x))=f_i(p)$, we get $f_i(p) \in \pi_M (A)$.  
\end{proof}

\begin{thm} \label{t:proj}
Take any SSP such that its fiber maps are diffeomorphisms of an arbitrary compact manifold (possibly with boundary).
Then the statistical attractor can be reconstructed from its projection onto the fiber in the following way: the point $x$ belongs to $\As$ if and only if the projections of all preimages of this point belong to the projection of~$\As$. This statement also holds for the Milnor attractor. 
\end{thm}
\begin{proof}
Take $x \in \As$. Since the attractor is invariant, all preimages of $x$ also belong to the attractor, so their projection onto the fiber lie in the projection of the attractor. 

Let us prove the inverse implication. Take a point $x = (\omega, p)$ such that all preimages of $x$ lie in the projection of the attractor and a number $n>0$. Then $F^{-n}(x) = (\sigma^{-n} \omega, q)$ where
\[
q = f_{\omega_{-n}}^{-1} \circ \dots \circ f_{\omega_{-1}}^{-1}(p).
\]
Since $q = \pi_M F^{-n}(x) \in \pi_M(\As)$, for some sequence $\tilde \omega$ we have $(\tilde \omega, q) \in \As$. By Corollary~\ref{c:future} we may assume that the future halves of $\tilde \omega$ and $\sigma^{-n} \omega$ coincide:
\[
\tilde \omega_0 = \omega_{-n}, \; \dots, \; \tilde \omega_{n-1} = \omega_{-1}, \; \dots
\]
Denote $y_n = (\tilde \omega, q)$. Then
\[
F^n(y_n) = (\sigma^n \tilde \omega, f_{\omega_{-1}} \circ \dots \circ f_{\omega_{-n}} (q))
= (\sigma^n \tilde \omega, p).
\]
The points $x$ and $F^n(y_n)$ have identical fiberwise coordinates, the future halves of their basewise coordinates coincide, and the past halves coincide up to the element with index $-n$. Thus, for $n \to \infty$ we have $F^n(y_n) \to x$. Since $\As$ is invariant, $F^n(y_n) \in \As$. Since the attractor is closed, we have $x \in \As$.

This proves Theorem~\ref{t:proj} for the statistical attractor. Similar statement for the Milnor attractor holds by Remark~\ref{r:AMAS}.
\end{proof}

\subsection{Stability of the projection} \label{s:stproj}
\begin{prop} \label{p:stproj}
For any neighborhood $U$ of the statistical attractor there exist a neighborhood~$U_M$ of its projection to the fiber and a number~$n \in \mathbb N$ such that $F^n(\Sig \times U_M) \subset U$.
\end{prop}
\begin{proof} 
Recall that the distance on $X$ between two points $(\omega_1, p_1)$ and $(\omega_2, p_2)$ is defined as the sum of distances along the base and along the fiber: $\dist_\Sig(\omega_1, \omega_2) + \dist_M(p_1, p_2)$. 
Consider a number $\eps>0$ such that the neighborhood~$U$ contains~$U_\eps(\As)$.

Let $U_M$ be the $\delta$-neighborhood of $\pi(\As)$ where the small number $\delta$ will be specified later. Consider an arbitrary point $x=(\omega, p) \in \Sig \times U_M$. Consider the point $x_A=(\omega_A, p_A)\in\As$ such that $\dist_M(p, p_A) < \delta$. By Corollary~\ref{c:future}, it does not depend on the future half of the base coordinate whether a point belongs to the attractor. Therefore, we can assume that the future halves of the sequences~$\omega$ and~$\omega_A$ coincide. This coincidence implies that we can find an integer $n$, independent of~$x$, such that $\dist_\Sig(\sigma^n(\omega), \sigma^n(\omega_A)) < \eps/2$. 
Let $g=f_{\omega_{n-1}} \circ \dots \circ f_{\omega_0}$. Then the projections to the fiber of the points $F^n(x)$ and $F^n(x_A)$ are $g(p)$ and $g(p_A)$ respectively. Chose a number $L>0$ in such a way that every fiber map be $L$-Lipschitz. Then the map~$g$ will be $L^n$-Lipschitz. Now we can set $\delta=\eps/(2 L^n)$. Then it follows from $\dist_M(p, p_A) < \delta$ that the fiberwise distance between $F^n(x)$ and $F^n(x_A)$ is less than $\eps/2$. Since those points are also $\eps/2$-close basewise, we conclude that they are $\eps$-close in the metric on~$X$. Since $F^n(x_A) \in \As$, we have $F^n(x) \in U_\eps(\As) \subset U$.
\end{proof}

\begin{thm} \label{stproj}
The statistical attractor is Lyapunov stable if and only if its projection to the fiber is Lyapunov stable in sense of Definition~\ref{d:stproj}. This statement also holds for the Milnor attractor. 
\end{thm}
\begin{proof}

First let us deduce the stability of projection from the stability of $\As$. Consider any neighborhood~$U_M$ of the projection of $\As$. Let us find a neighborhood~$U$ of the attractor such that $U\subset\Sig \times U_M$. Since the attractor is stable, it admits a neighborhood~$V$ whose images are contained $U$. Applying Proposition~\ref{p:stproj} to $V$, we get $V_M$ and $n$ such that $F^n(\Sig \times V_M) \subset V$. Thus for any~$N \ge n$ the set $F^N(\Sig \times V_M)$ is a subset of~$U$, and therefore a subset of $\Sig \times U_M$. Truncating the neighborhood~$V_M$ if necessary, one can assure that $F^N(\Sig \times V_M) \subset \Sig \times U_M$ for $N = 0, \dots, n-1$ also. The stability of the projection is proven.

Suppose now that the projection of $\As$ is stable. To prove that the attractor itself is stable, we need, given a neighborhood $U \supset \As$, to be able to construct a neighborhood~$V \supset \As$ such that every trajectory that starts inside~$V$ does not quit~$U$.

For that end, using Proposition~\ref{p:stproj}, let us find for our~$U$ a neighborhood~$U_M \supset \pi_M(\As)$ and a number~$n$ such that 
\begin{equation} \label{e:toA}
F^n(\Sig \times U_M) \subset U.
\end{equation}
By Definition~\ref{d:stproj}, there is a neighborhood~$V_M \supset \pi_M(\As)$ such that
\[
\forall k \ge 0 \quad F^k(\Sig \times V_M) \subset \Sig \times U_M.
\]
Therefore,
\begin{equation} \label{e:stproj}
\quad F^{n+k}(\Sig \times V_M) \subset F^n(\Sig \times U_M).
\end{equation}

Consider $V=\Sig \times V_M$.  Due to~\eqref{e:toA} and~\eqref{e:stproj}, for any $m \ge n$ we have $F^m(V) \subset U$. Truncating~$V$ if necessary, we can assure that this also holds for $m < n$.

Similar statement for the Milnor attractor holds by Remark~\ref{r:AMAS}.
\end{proof}

%% file: circle.tex
\section{Stability of attractors for SSPs with the cirle fiber} \label{s:stable}

Fix integers $r, s \in \mathbb N, \; s \ge 2$. 
Let $T~=~(\Diff_+^r(\ci))^s$ be a set of all SSPs over the Bernoulli shift with the fiber~$\ci$ such that the fiber maps~$f_1, \dots, f_s$ are orientation preserving $C^r$-diffeomorphisms of a circle. 

Recall that a subset~$R$ of a topological space~$T$ is called \emph{residual} if it contains a countable intersection of everywhere dense open sets; one usually also assumes that~$T$ is \emph{a Baire space}, i.e., that any residual subset is everywhere dense in~$T$.
Let us prove that~$T$ is a Baire space. The space~$(C^r(\ci))^s$ is complete, and by the Baire theorem any complete metric space is a Baire space. On the other hand, $T$ is an open subset of this space, and any open subset of a Baire space is itself a Baire space.

\begin{thm} \label{t:stable}
There is a residual set $R \subset T$ such that for any SSP from~$R$ the statistical attractor is Lyapunov stable and coincides with the Milnor attractor. 
\end{thm}
\begin{proof}
We will consider only SSPs for which the fiber map~$f_1$ is a Morse--Smale diffeomorphism (i.e., $f_1$ has finitely many periodic orbits, and these orbits are hyperbolic). Denote by $S \subset T$ the set of SSPs with this property. Since Morse--Smale diffeomorphisms form an open and everywhere dense subset of $\Diff_+^r(\ci)$, the set $S \subset T$ is open and everywhere dense. 

Let $\attr$ be the set of all sinks of the fiber map~$f_1$. For $a \in \attr$ denote by $\orb(a)$ the closure of the orbit of the point~$a$ under the action of the semigroup generated by the fiber maps $f_1, \dots, f_s$.

Define $R \subset S$ as the set of all SSPs for which the following genericity conditions hold: 
\begin{enumerate}\setlength\itemsep{-0.1em}
\item \label{i:f1f2}
for any periodic point $y$ of the map $f_1$, the point $f_2(y)$ is not a periodic point of $f_1$;
\item \label{i:stable}
for any $a \in \attr$ the set $\orb(a)$ is Lyapunov stable in the sense of Definition~\ref{d:stproj}.
\end{enumerate}
It is clear that the first condition provides an open and everywhere dense subset of~$S$. The second one defines a residual subset of of $S$, which will be proved in section~\ref{s:orb} below. Since $S \subset T$ is open and dense, $R$ is a residual subset of~$T$. 

Now let us consider an arbitrary SSP $F \in R$ and prove that its attractor is stable.
Let $\pi_\ci: X \to \ci$ be the projection to the fiber, as usual.

\begin{prop}
For a $\mu_X$-a.e. point $x \in X$ the set $\pi_\ci(\oms(x))$ coincides with $\orb(a)$ for some $a \in \attr$.
\end{prop}
\begin{proof}
Given points $x, y \in \ci$, we will write $x \to y$ if for any neighborhood~$U$ of the point~$y$ there exists a composition~$g$ of the fiber maps~$f_i$ such that $g(x) \in U$.

It follows from the first genericity condition that for any point $y \in \ci$ there is a sink $a \in \attr$ such that $y \to a$. Indeed, if~$y$ is not a repeller of the map~$f_1$,  then when $f_1$ is applied iteratively $y$ is attracted to one of the periodic attractors of~$f_1$. If~$y$ is a repeller, then the same argument can be applied to the point~$f_2(y)$.

By Corollary~\ref{c:inv}, for a $\mu_X$-a.e. point $x \in X$ the set $\pi_\ci(\oms(x))$ is forward-invariant under the fiber maps. Fix such a point~$x$ and let $\Omega=\pi_\ci(\oms(x))$.
Take an arbitrary point $y \in \Omega$ and find a sink $a \in \attr$ such that $y \to a$. Let us prove that $\Omega = \orb(a)$.

Since the set $\Omega$ is invariant, it follows from~$y \to a$ that there are points of $\Omega$ arbitrarily close to~$a$. Since $\Omega$ is closed, this implies $a \in \Omega$. Applying the fact that $\Omega$ is closed and invariant once again, we get $\orb(a) \subset \Omega$.

To prove the inverse inclusion $\Omega \subset \orb(a)$, we will use the Lyapunov stability (in the sense of Definition~\ref{d:stproj}) of the set $\orb(a)$; recall that this stability is provided by the second genericity condition. We will argue by contradiction. Suppose that $z \in \Omega \setminus \orb(a)$. Consider a neighborhood $U \supset \orb(a)$ that does not contain~$z$. Given $U$, find a neighborhood $V \supset \orb(a)$ such that all trajectories that start inside $\Sig \times V$ never leave $\Sig \times U$. But this is in contradiction with the fact that the set $\oms(x)$ intersects $\Sig \times V$ and contains~$z$.
\end{proof}

Thus, for almost every point $x$ we have $\pi_\ci (\oms(x)) = \orb(a)$ for some $a \in \attr$. It follows from Proposition~\ref{p:defPAM} that the projection of the statistical attractor is the union of several sets of the form $\orb(a), \; a \in \attr$, namely, those that coincide with $\pi_\ci (\oms(x))$ for a set of points $x$ that has positive measure. Hence, the projection of the statistical attractor to the fiber is stable, which by Theorem~\ref{stproj} implies that $\As$ itself is stable. If the statistical attractor is Lyapunov stable, it always coincides with the Milnor attractor; the proof of this simple fact can be found in~\cite[Lemma~$5.3$]{Oku}.  
\end{proof}

Now let us deduce from Theorem~\ref{t:stable} an analogous statement for the SSPs with the segment fiber.
Denote by $\Diff_+^r(I)$ the set of orientation preserving $C^r$-smooth maps of a segment~$I$ inside itself that are diffeomorphisms onto the image. Let $T_I = (\Diff_+^r(I))^s$ be the space of all SSPs with the fiber a segment and with $s$ fiber maps, where $s \ge 2$ is arbitrary.

\begin{cor} \label{c:interval}
There exists a residual set $R_I \subset T_I$ such that for any SSP from $R_I$ the statistical attractor is Lyapunov stable and coincides with the Milnor attractor. 
\end{cor}

\begin{proof}
Let us regard the segment $I$ as an arc of a circle: $I \subset \ci$. Let $\tilde T \subset T$ be the set of SSPs with the fiber a circle such that all fiber maps send the arc~$I$ strictly into itself, and $\tilde R \subset \tilde T$ be the intersection of $\tilde T$ with the residual subset $R$ from Theorem~\ref{t:stable}.

Consider the map $\Pi\colon \tilde T \to T_I$ that maps a SSP with the fiber a circle into its restriction to $\Sig\times I$. According to~\cite[Lemma~$4.25$]{Phe}, any continuous open surjection from a complete metric space onto a Hausdorff space takes residual sets to residual sets. Thus, the set $R_I := \Pi(\tilde R) \subset T_I$ is residual. 

Consider an arbitrary SSP with the fiber a segment $F \in R_I$. It can be extended as an SSP with the fiber a circle $\tilde F \in \tilde R$. By Theorem~\ref{t:stable} the statistical attractor of~$\tilde F$ is Lyapunov stable and coincides with the Milnor attractor. Since the statistical and the Milnor attractors of the SSP~$F$ can be obtained as the intersections of the corresponding attractors of~$\tilde F$ to the dissipative set~$\Sig\times I$, the statistical attractor of~$F$ is Lyapunov stable and coincides with the Milnor attractor too.
\end{proof}

%% file: semicont.tex
\subsection{Sink orbits closures are stable}\label{s:orb}
\begin{lem} \label{l:stable}
Skew products $F \in S$ such that for any $a \in \attr$ the set $\orb(a)$ is Lyapunov stable form a residual subset of~$S$.
\end{lem}

The proof is in many respects similar to the proof of Theorem~$6.1$ from~\cite{MP}. It is based on the well-known semicontinuity lemma. Let us recall the statement of this lemma.

\begin{defin} A map $O$ from a metric space $T$ to the set of all closed subsets of a compact space~$X$ is called \emph{lower semi-continuous} at the point $F \in T$ if for any $\eps$ there is $\delta$ such that if $d(F, F')<\delta$, then for every point $x \in O(F) \subset X$ in the set $O(F')$ there is a point $\eps$-close to~$x$.
\end{defin}

\begin{lem} (The semicontinuity lemma)
\label{Baire}
The set of continuity points (with respect to the Hausdorff metric on the image) of a lower semi-continuous map is residual.
\end{lem}
This lemma is well-known; its proof can be found in, e.g.,~\cite{Kur} and~\cite{San}. 

\bigskip
Let us prove Lemma~\ref{l:stable} now. 
\bi
\item
Consider an arbitrary SSP $F \in S$ and its neighborhood $W \subset S$ such that in $W$ all periodic points of the fiber map~$f_1$ survive and no new periodic points appear. Such neighborhood exists because the map~$f_1$ is a Morse--Smale diffeomorphism, so it is structurally stable. 

\item
Denote by $M$ the set of all closed subsets of $\ci$ with the Hausdorff metric.
For each sink $a_i \in \attr$ of the map~$f_1$ consider the map $O_i: W \to M$ that puts in correspondence to the SSP $\tilde F$ the set $\orb_{\tilde F}(\tilde a_i)$, where~$\tilde a_i$ stands for the continuation for~$\tilde F$ of the periodic point~$a_i$ of the map~$F$.
\ei

\begin{prop}
\label{semicont}
The map $O_i : W \to M$ is lower semi-continuous.
\end{prop}
\begin{proof}
Given arbitrary $\eps$, the set $O_i(F)=\orb(a_i)$ admits a finite $\eps$-net that consists of points of $\mathrm{orb}(a_i)$. Suppose this $\eps$-net consists of points $w_1(a_i), \dots, w_m(a_i)$, where $w_j$ are finite compositions of the fiber maps. For any sufficiently close SSP~$\tilde F$ the points $\tilde w_1(\tilde a_i), \dots, \tilde w_m(\tilde a_i)$ are shifted by less than $\eps$ and form a $2 \eps$-net for $O_i(F)$. But all these points belong to $O_i(\tilde F)$, and therefore in the vicinity of any point $O_i(F)$ there are points of~$O_i(\tilde F)$.
\end{proof}

\begin{prop}\label{locres}
SSPs $\tilde F$ for which all sets $\orb_{\tilde F}(\tilde a_i)$ are Lyapunov stable form a residual subset $R_W$ of $W$.
\end{prop}

\begin{proof}
The map $O_i$ is lower semi-continuous by Proposition~\ref{semicont}. Hence, by Lemma~\ref{Baire} it is continuous on some residual subset $R_i \subset W$. Proposition~\ref{stcont} below says that if the set $\orb(\tilde a_i)$ is Lyapunov unstable, then it does not continuously depend on the map. Therefore, for any SSP from~$R_i$ the set $\orb(\tilde a_i)$ is Lyapunov stable. It suffices to set $R_W$ equal to the intersection of all~$R_i$.
\end{proof}

Lemma~\ref{l:stable} can be obtained from Proposition~\ref{locres} using the following trivial statement:

\begin{prop}
Suppose that every point $x$ of a separable metric space~$S$ admits a neighborhood $U(x)$ such that the set $B \subset S$ intersects $U(x)$ by a residual subset of $U(x)$. Then~$B$ is residual in $S$. 
\end{prop}
\begin{proof}
Consider a countable everywhere dense subset and for each its point take its neighborhood~$U_n$ such that $B\cap U_n$ is residual. Then the set $U=\bigcup U_n$ is open and everywhere dense. Let $B_n=U \setminus (U_n \setminus B)$. It is easy to see that $B_n$ is residual. Now, $\bigcap B_n \subset B$, i.e.,~$B$ contains a countable intersection of residual sets which is itself residual.
\end{proof}

\begin{prop} \label{stcont} 
For $F \in S$, if the set $\orb(a_i)$ is Lyapunov unstable, then the map~$O_i$ is discontinuous at~$F$.
\end{prop}

\begin{proof}
Denote $a=a_i$.
Let us suppose that the set $\orb(a)$ is Lyapunov unstable and prove that then it depends on the map discontinuously. Instability implies that there exists an open set~$U$ such that $U$ does not intersect $\orb(a)$ and for any $\delta$ there are a point~$x$ and a composition $v$ of the fiber maps such that 
\begin{equation} \label{e:runaway}
x \in U_{\delta}(\orb(a)), \; v(x) \in U.
\end{equation}  
To show that the map $O_i$ is discontinuous at $F$, it is sufficient to construct, for any given~$\delta$, a map $G$ that is $2\delta$-close to~$F$ in $(\Diff^r)^s$--metric and such that the set $\orb_G(a_G)$ intersects~$U$. From this point on, for an SSP $G$ close to $F$, the lower index~$G$ in the notation for some object related to~$F$ means that we are considering the analogous object for~$G$.

Let us find, given an arbitrary $\delta$, a point $x$ and a word $v$ such that~\eqref{e:runaway} holds. Since $\oorb(a)$ is dense in $\orb(a)$, there exists a composition of fiber maps~$w$ such that $\dist(w(a), x) < 2 \delta.$ Below we will find an SSP $G$ that is $2\delta$-close to $F$ and such that 
\begin{equation} \label{e:discontinuous}
v_G(w_G(a_G)) = v(x).
\end{equation}
Then, since $v(x) \in U$, the point $v_G(w_G(a_G))$ will be the required point that belongs to $\orb_G(a_G)$, but lies inside~$U$. The proof will be complete when we find such a map~$G$.

\emph{The construction of $G$.} Choose arbitrary lifts $\tilde f_i$ of the fiber maps of~$F$ from the circle $\RR/\mathbb Z$ to the line~$\RR$. The choice of these lifts will define the lifts $\tilde v$ and $\tilde w$ of the maps $v$ and $w$. Choose the lifts $\tilde a$ and $\tilde x$ of the points $a$ and $x$ in such a way that $|\tilde x - \tilde w(\tilde a)| < 2 \delta$. Without loss of generality we may assume that 
\begin{equation} \label{e:btw}
\tilde w(\tilde a) < \tilde x < \tilde w(\tilde a)+2\delta.
\end{equation} 

Recall that our fiber maps are orientation preserving, i.e., since the fiber is one-dimensional, monotonically increasing. Consider a perturbed SSP
$$F_c=\{f_1+c, \dots, f_k+c\}, \; c \in [0, 2\delta].$$   

When $c=0$, one has $\tilde v_{F_0}(\tilde w_{F_0}(\tilde a_{F_0}))<\tilde v(\tilde x)$. 

Now let $c=2\delta$. Then $\tilde a_{F_{2 \delta}} \ge \tilde a$, because when lifting the graph of a monotonically increasing function the sink also moves upwards. 
Since $\tilde w_{F_{2 \delta}} \ge \tilde w+2\delta$, we have $\tilde w_{F_{2 \delta}}(\tilde a_{F_{2 \delta}}) \ge \tilde x$. 
Since $\tilde v_{F_{2 \delta}} \ge \tilde v+2\delta \ge \tilde v$, the previous inequality implies $\tilde v_{F_{2 \delta}}(\tilde w_{F_{2 \delta}}(\tilde a_{F_{2 \delta}})) \ge \tilde v(\tilde x)$.

By the intermediate value theorem, for some number $\hat c \in [0, 2 \delta]$ we will have $\tilde v_{F_{\hat c}}(\tilde w_{F_{\hat c}}(\tilde a_{F_{\hat c}})) = \tilde v( \tilde x)$. Hence the SSP $G=F_{\hat c}$ satisfies~\eqref{e:discontinuous}.
\end{proof}